\documentclass{eptcs}

\title{Categorical semantics of a \\ simple differential programming language}
\def\titlerunning{Categorical semantics of a simple differential programming language}
\author{Geoffrey Cruttwell
\institute{Mount Allison University, Sackville, Canada}
\email{gcruttwell@mta.ca}
\and
Jonathan Gallagher  \qquad\qquad   Dorette Pronk
\institute{Dalhousie University, Halifax, Canada}
\email{jonathan.gallagher@dal.ca \qquad\qquad  dorette.pronk@dal.ca}}
\def\authorrunning{G. Cruttwell, J. Gallagher, D. Pronk}

\providecommand{\event}{ACT 2020}

\usepackage{etex}

\usepackage{mathrsfs}

\usepackage{hyperref}

\usepackage{comment}

\usepackage[all,2cell,cmtip]{xy}
\UseTwocells

\usepackage{graphicx}
\usepackage[dvipsnames]{xcolor}
\usepackage{enumerate}
\usepackage{BarrTo}
\usepackage{amsfonts,amsmath,amssymb,amsthm}
\usepackage{proof}
\usepackage{tikz-cd}

\usepackage{tikz}
\usetikzlibrary{decorations.pathmorphing}
\tikzset{snake it/.style={decorate, decoration=snake}}
\usetikzlibrary{arrows,calc} 
\usepackage{tikzit}

\tikzstyle{red}=[fill={rgb,255: red,255; green,0; blue,4}, draw=black, shape=circle]
\tikzstyle{green}=[fill={rgb,255: red,25; green,255; blue,0}, draw=black, shape=circle]
\tikzstyle{square}=[fill=white, draw=black, shape=rectangle]
\tikzstyle{circle}=[fill=white, draw=black, shape=circle]

\usepackage[all]{xy}

\usepackage{microtype}

\usepackage{stmaryrd}

\newcommand{\ty}{\mathsf{Ty}}
\newcommand{\real}{\mathsf{real}}

\newcommand{\true}{\mathsf{true}}
\newcommand{\fd}{\mathsf{fd}}
\newcommand{\false}{\mathsf{false}}
\newcommand{\sdpl}{\mathsf{SDPL}}

\newcommand{\rd}{\mathsf{rd}}
\newcommand{\Rd}{\mathcal{R}}

\newcommand{\op}{\text{op}}
\newcommand{\total}{\mathsf{total}}
\newcommand{\fst}{\mathsf{fst}}
\newcommand{\snd}{\mathsf{snd}}

\newcommand{\abun}{\mathsf{ABun}}

\newcommand{\lin}{\mathsf{Lin}}

\newcommand{\R}{\mathsf{R}}

\newcommand{\ev}{\mathsf{ev}}
\newcommand{\x}{\times}

\newcommand{\blank}{\underline{~}}

\newcommand{\X}{\mathbb{X}}

\newcommand{\bev}{\mathsf{bev}}

\newcommand{\proves}{\vdash}
\newcommand{\<}{\left\langle}
\renewcommand{\>}{\right\rangle}

\newcommand{\smooth}{\mathsf{Smooth}}

\renewcommand{\lim}{\mathsf{lim}}

\newcommand{\ex}{\mathsf{ex}}

\newcommand{\bun}{\mathsf{bun}}

\newcommand{\C}{\underline{C}}

\newcommand{\fun}{\mathsf{Fun}}

\newcommand{\fv}{\mathsf{fv}}

\newcommand{\M}{\mathcal{M}}

\newcommand{\den}[1]{{\left\llbracket {#1} \right\rrbracket}}

\newcommand{\rs}[1]{\, \overline{ {#1} } \,}

\newcommand{\difft}[2]{\frac{\partial {#1}}{\partial {#2}}}

\newtheorem{theorem}{Theorem}[section]

\newtheorem{proposition}[theorem]{Proposition}
\newtheorem{corollary}[theorem]{Corollary}

\newtheorem{lemma}[theorem]{Lemma}

\newtheorem{definition}[theorem]{Definition}

\theoremstyle{definition}

\begin{document}

\maketitle 

\section{Introduction}

With the increased interest in machine learning, and deep learning in particular 
(where one extracts progressively higher level features from data using multiple layers of processing), 
the use of automatic differentiation has become more wide-spread in computation.
See for instance the 
surveys given in \cite{book:deep-learning} and \cite{journal:Baydin-diff-machine-learning}.  In fact, Facebook's Chief AI scientist Yann LeCun has gone as far as famously exclaiming:
\begin{quote}
  ``Deep learning est mort. Vive Differentiable Programming!  ...people are now building a new kind of 
  software by assembling networks of parameterized functional blocks and by training them from examples 
  using some form of gradient-based optimization.''\footnote{\href{https://www.facebook.com/yann.lecun/posts/10155003011462143}{Facebook post on Jan 5, 2018}}
\end{quote}
The point being that differentiation is no longer being viewed as merely a useful tool when creating software, but 
instead becoming viewed as a fundamental building block.  This sort of ubiquity warrants a more in-depth study of automatic 
differentiation with a focus on treating it as a fundamental component.

%

There have been two recent developments to provide the theoretical support for this type of structure.  In fact, the settings described above use two types of differentiation: the usual forward derivative to analyse the effect of changes in the data, as well as the reverse derivative to allow for error correction (i.e., training) through the efficient calculation of the gradients of functions.  Thus, any theoretical approach must be able to deal with both types of differentiation.    One approach is presented in \cite{journal:Abadi-Plotkin}, where Abadi and Plotkin provide a simple differential programming language with conditionals, recursive function definitions, and a notion of reverse-mode differentiation (from which forward differentiation can be derived) together with both a denotational and an operational semantics, and theorems showing that the two coincide.
Another approach is given in \cite{arxiv:RDC}, where the authors present reverse differential categories, a categorical setting for reverse differentiation.  They also show how every reverse differential category gives rise to a (forward) derivative and a canonical ``contextual linear dagger'' operation. The converse is true as well: a category with a foward derivative (that is, a Cartesian differential category \cite{journal:BCS:CDC}) with a contextual linear dagger has a canonical reverse derivative.

In the present paper we bring these two approaches together.  In particular, we show how an extension of 
reverse derivative categories models Abadi and Plotkin's language, and describe how this categorical model 
allows one to consider potential improvements to the operational semantics of the language.  
To model Abadi and Plotkin's language categorically, reverse derivative categories are not sufficient, 
due to their inability to handle partial functions and control structures.  
Thus, we need to add partiality to reverse differential categories.  The standard categorical machinery
to model partiality is restriction structure, which assigns to each map a partial 
identity map, subject to axioms as described in \cite{journal:rcats1}.  Combining 
this structure with reverse differential structure, 
we introduce {\em reverse differential restriction categories}. 
In addition to the list of axioms {\bf [RD.1 -- 7]} given for reverse differential categories 
in \cite{arxiv:RDC}, we require two additional axioms expressing how the restriction of the 
reverse derivative of a function is related to the restriction of the function itself and what the 
reverse derivative of a restriction of a function needs to be (cf.~Definition \ref{dfn:RDRC}).
The results characterizing the relationship between differential and reverse differential categories 
in terms of a contextual linear dagger extend to the context of restriction categories.  We also get 
for free that the reverse derivative preserves the order on the maps and preserves joins of maps, if they exist.

In Section \ref{sec:interpretation} we show how Abadi and Plotkin's language can be 
modelled in a  reverse differential restriction category. We do this in two steps: at first we modify 
their language by omitting general recursion and instead only include while-loops. 
While-loops can be modelled in terms of recursion, but by separating them out we can see 
that source-transformation techniques (not discussed explicitly in Abadi and Plotkin 
but used in some commercial systems such as Theano \cite{journal:theano-ref}, TensorFlow \cite{journal:tensorflow-ref},
 and Tangent \cite{journal:tangent-ref}) always hold in our semantics (see Section \ref{section:source-transformation}); source 
 transformation techniques are not used for general recursion.   We also note that in order to be 
 able to push 
differentiation through the control structure, Abadi and Plotkin need that for each predicate  
the inverse images of {\em true} and {\em false} are both open sets. In the context of restriction 
categories we model this instead by providing two partially defined maps into the terminal object $1$ 
for each predicate symbol (one for true and one for false) with the requirement that their restrictions 
do not overlap (cf.~Section \ref{section:cat-interp-sdpl}). 

In the process of modelling Abadi and Plotkin's language, we see that not all our axioms are needed. 
Specifically,  Axioms {\bf [RD.6]} and {\bf [RD.7]} of a reverse differential restriction category (which deal with the behaviour of repeated reverse derivatives)
are not strictly necessary to model Abadi and Plotkin's language.  
However, in the final section of the paper, we show that if these axioms are present, 
changes can be made to the operational semantics to improve the efficiency and applicability of the simple differential programming language.  

Abadi and Plotkin's language represents an approach that makes the reverse derivative a language primitive
in a functional language.  Other approaches have been proposed to use reverse-mode accumulation for computing the derivative in a functional language.  Given a function $\R^n \to^{f} \R^m$, Pearlmutter and Siskind 
\cite{journal:lambda-as-backprop} discuss how to compute the Jacobian matrix of $f$ in a functional 
language by performing transformations on the function's computational graph.  This idea is similar 
to the symbolic differentiation of trace or tape terms in Abadi and Plotkin's language.  Elliot \cite{elliott-AD-ICFP}
shows how to view this sort of reverse-mode accumulation using continuations: when a function, written as a 
composition of simple operations, is written in continuation passing style, the reverse derivative of its computation
graph corresponds to a sort of generalized derivative of the continuation.  
In \cite{journal:rev-diff-shift-reset}, Wang et al extend Pearlmutter and Siskind's work by showing that 
the move to continuations allows for getting around the \emph{nonlocality} issue in the earlier work. 
Brunel et al \cite{journal:diff-prog-linear-logic} extend Wang's work from the point of 
view of linear logic, and allow for additional analyses based on tracking the linearity of a variable.
Abadi and Plotkin's work represents a next step in this area by considering, in addition, control flow 
structures and general recursive functions.

This current work contrasts to work submitted to ACT 2019 on modelling differential programming using 
synthetic differential geometry (SDG) (see \cite{book:kock-sdg,book:Lavendhomme} for an introduction to 
SDG).  In the previous work a simple differential programming language 
featuring forward differentiation was introduced and an interpretation into a well-adapted model of 
SDG was given (see e.g. \cite{journal:cahiers-topos-dubuc} for such models).  The focus was on exploring 
what programming languages features might be able to exist soundly with differential programming.  The 
current work develops the categorical semantics of Abadi and Plotkin's language for reverse differentiation as well as the categorical 
semantics of source-transformations for their language.  In particular we show that the operational semantics 
is modelled soundly by a denotational semantics into our categorical framework.  We will also see that 
using the axiomatic approach developed here leads to a sound exponential speedup when computing the 
reverse derivative of looping-phenomena.

\section{Background: Relevant Categorical Structures}

In this section, we briefly review some of the relevant structures from category theory which we will make use of.

\subsection{Cartesian and reverse differential categories}

The canonical category for differentiation is the category $\smooth$ whose objects are the powers of the reals $\R$ ($\R^0 = \{1\},\R, \R^2, \R^3$, etc.) and whose maps are the smooth (infinitely differentiable) maps between them.  To any map $f: A \to B$ in this category, there is an associated map 
	\[ D[f]: A \times A \to B \]
whose value at $(x,v) \in A \times A$ is $J(f)(x)\cdot v$, the Jacobian of $f$ at $x$, taken in the direction $v$.  This map satisfies various rules; for example, the chain rule is equivalent to the statement that for any maps $f: A \to B, g: B \to C$,
	\[ D[fg] = \<\pi_0f,D[f]\>D[g]. \]
(Note that we use the path-order for composition, so $fg$ means ``first $f$ then $g$''.)
Many other familiar rules from calculus can be expressed via $D$; for example, the symmetry of mixed partial derivatives can be expressed as a condition on $D^2[f] = D[D[f]]$.  

\begin{definition} (\cite[Defn. 2.1.1]{journal:BCS:CDC}) A \textbf{Cartesian differential category} or \textbf{CDC} is a Cartesian left additive category (\cite[Defn. 1.3.1]{journal:BCS:CDC}) which has, for any map $f: A \to B$, a map 
	\[ D[f]: A \times A \to B \]
satisfying seven axioms $\textbf{[CD.1--7]}$.
\end{definition}
The formulation of CDCs and indeed the other flavours of categories with derivatives we will use have the intent that 
in $\<a,v\>D[f]$, $a$ is the point and $v$ is the direction; this is in contrast to the original formulation of CDCs 
which had the point and direction swapped, and we chose the point-direction formulation because most of the literature 
follows this convention.  This causes a change to axioms \textbf{CD.2,6,7}.

While $\smooth$ is the canonical example, there are many others, including examples from infinite dimensional vector spaces, synthetic differential geometry, algebraic geometry, differential lambda calculus, etc: see \cite{journal:BCS:CDC, thesis:myphdthesis,  journal:Diff-Bund,  journal:Cockett-Seely:Faa, journal:TangentCats}.  

In contrast, the reverse derivative, widely used in machine learning for its efficiency, is an operation which takes a smooth map $f: A \to B$ and produces a smooth map
	\[ R[f]: A \times B \to A \]
whose value at $(x,w) \in A \times B$ is $J(f)^{T}(x)\cdot w$, the transpose of the Jacobian of $f$ at $x$, taken in the direction $w$.  

There are two possible ways to categorically axiomatize the reverse derivative.  One way is to start with a CDC and ask for a dagger structure (representing the transpose); one could then use the dagger with the $D$ from the CDC to define a reverse derivative $R$.  However, there is some subtlety in this: the dagger structure is only present on the \emph{linear} maps of the category, not on all the maps of the category. The other way is to axiomatize $R$ directly, as was done in \cite{arxiv:RDC}.  

\begin{definition} (\cite[Defn. 13]{arxiv:RDC}) A \textbf{reverse differential category} or \textbf{RDC} is a Cartesian left additive category which has, for any map $f: A \to B$, a map
	\[ R[f]: A \times B \to A \]
satisfying seven axioms.
\end{definition}

For example, in this formulation the chain rule is equivalent to {\bf [RD.5]}, the rule that for any maps $f: A \to B$, $g: B \to C$,
	\[ R[fg] = \<\pi_0, (f\x 1)R[g]\>R[f]. \]
	
Moreover, there is something striking about a reverse differential structure: any RDC is automatically a CDC.  If one applies the reverse derivative twice, 0's out a component and projects, the result is the forward derivative.   That is, the following defines a (forward) differential structure from a reverse differential structure (see \cite[Theorem 16] {arxiv:RDC}):
  \[
    \infer{
      D[f]:= A\x A \to^{\<\pi_0,0,\pi_1\>} (A\x B)\x A \to^{R[R[f]]} A\x B \to^{\pi_1} B
    }{
      \infer{
        (A\x B) \x A \to^{R[R[f]]} A\x B 
      }{
        \infer{
          A\x B \to^{R[f]} A
        }{
          A \to^{f} B
        }
      }
    }  
  \]
  
Thus, while a ``dagger on linear maps'' is required to derive an RDC from a CDC, no such structure is required to go from an RDC to a CDC.  In fact, one can show that a CDC with a ``dagger on linear maps'' is equivalent to an RDC: see Theorem 42 in \cite{arxiv:RDC}.  

For this reason, as well as the fact that the reverse derivative is of greater importance in machine learning, in this paper we take a reverse differential category to be the primary structure.  

\subsection{Restriction categories and differential restriction categories}

Of course, to model a real-world programming language which involves non-terminating computations, we must also be able to handle partial functions.  For this, we turn to restriction categories \cite{journal:rcats1}, which allow one to algebraically model categories whose maps may only be partially defined.  Consider the category of sets and partial functions between them.  To any map $f: A \to B$ in this category, there is an associated ``partial identity'' map $\rs{f}: A \to A$, which is defined to be the identity wherever $f$ is defined, and undefined otherwise.  This operation then has various properties such as $\rs{f} f = f$.  This is then axiomatized:

\begin{definition} (\cite[Defn. 2.1.1]{journal:rcats1}) A \textbf{restriction category} is a category which has for any map $f: A \to B$, a map $\rs{f}: A \to A$ satisfying various axioms.
\end{definition}

In section \ref{sec:RDRC}, we will combine restriction structure with reverse differential structure to get the categorical structure we will use to model Abadi and Plotkin's language.  

Before we get to that, however, we will need to briefly review a few definitions from restriction category theory.  It will also be helpful to consider the previously defined combination of restriction structure and (forward) differential structure.

A restriction category allows one to easily talk about when a map is ``less than or equal to'' a parallel map and when two parallel maps are ``compatible'':

\begin{definition}
Suppose $f,g: A \to B$ are maps in a restriction category.  Write $f \leq g$ if $\rs{f} g = f$, and write $f \sim g$ (and say ``$f$ is compatible with $g$'') if $\rs{f}g = \rs{g}f$.  
\end{definition}
That is, $f \leq g$ if $g$ is defined wherever $f$ is defined, and when restricted to $f$'s domain of definition, $g$ is equal to $f$; $f \sim g$ if $f$ and $g$ are equal where they are both defined.  One can show that $\leq$ is a partial order on each hom-set; in fact,
restriction categories are canonically partial order enriched by $\leq$.  

Being able to ``join'' two compatible maps will be important when we define control structures such as ``if'' and ``while'', as we will being able to discuss when maps are ``disjoint''.  

\begin{definition}

\begin{enumerate}
\item  If $f,g: A \to B$ and there is a least upper bound $f\vee g$ with respect to the partial order defined above, we call $ f\vee g$ the \textbf{join} of $f$ and $g$. Note that this implies that $f$ and $g$ are compatible.
  \item The notion of join extends to families of maps that are pairwise compatible, and we write $\vee_i f_i$
to denote the join of the pairwise compatible family.
  \item Say that a map $\emptyset: A \to B$ is \textbf{nowhere defined} if $\emptyset$ is the minimum in the partial order.  
  \item Say that $f, g: A \to B$ are \textbf{disjoint} if $\rs{f}g$ is nowhere defined.  Any two disjoint maps are compatible. 
\end{enumerate}
\end{definition}

The formalization of disjoint joins in a restriction category was given in \cite{journal:cockett-manes-boolean-classical}
as part of the story of formalizing Hoare semantics in a \emph{classical} restriction category.    Further analysis of joins in restriction categories was provided in \cite{journal:guo-range-join}.
Giles \cite{phd:giles-calgary} used disjoint joins in connecting restriction categories to the semantics of 
reversible computing.  Disjoint 
joins in partial map categories correspond to disjoint joins of monics, which often give a coproduct (e.g. as 
in coherent categories).  One way to model iteration is to have a traced coproduct, and this can be 
directly expressed using disjoint joins: this approach was used in formalizing iteration in restriction categories and to 
build a partial combinatory algebra by iterating a step-function in \cite{journal:timed-sets,journal:total-maps}.
The formalization of iteration using disjoint joins was based on the work of Conway \cite{book:conway-iteration}.
Another approach to formalizing the semantics of iterative processes in a category using algebraic formalizations was 
introduced in \cite{chapter:iteration-theory}, refined in \cite{book:Bloom-Esik}, and 
further developed categorically in \cite{journal:elgot-algebra}.

Finally, it is worth noting that there has been previous work combining CDC structure with restriction structure \cite{journal:diff-rest}.  The canonical example of such a category is the category of smooth \emph{partial} maps between the $\R^n$'s.  The partiality acts in a compatible way with the derivative, as $D[f]: A \times A \to B$ is entirely defined in the second (vector) component: that is, the only partiality $D[f]$ has is from $f$ itself.  Thus, in a ``differential restriction category'', one asks that $\rs{D[f]} = \rs{f} \times 1$.  
These are formulated on top of the notion of \emph{cartesian left additive restriction category}: these are 
restriction categories with restriction products (which is a lax notion of product for 
restriction catgories developed in \cite{cockett_lack_2007}) and where each homset is a commutative 
monoid such that $x (f+g) = xf + xg$ and $0f \leq 0$, and finally projections fully preserve addition.
The intuition comes from considering partial, smooth functions on open subsets of $\R^n$: not all 
smooth functions preserve addition, but smooth functions are addable under pointwise addition.

\begin{definition} (\cite[Defn. 3.18]{journal:diff-rest}) A \textbf{differential restriction category} is a Cartesian left additive restriction category, which has, for each map $f: A \to B$, a map 
	\[ D[f]: A \times A \to B \]
satisfying various axioms \footnote{There are nine equational axioms mirroring the axioms for reverse differential restriction categories given in the sequel.}, including \textbf{[DR.8]}: $\rs{D[f]} = \rs{f} \times 1$.
\end{definition}

\section{Reverse differential restriction categories}\label{sec:RDRC}

We are now ready to define the new structure which we will use to model Abadi and Plotkin's language.  

\begin{definition}\label{dfn:RDRC}
  A {\bf reverse differential restriction category} or \textbf{RDRC} is a Cartesian left additive restriction category which has an operation on maps:
  \[
    \infer{A\x B \to_{R[f]} A}{A\to^{f} B}  
  \]
  such that
  \begin{enumerate}[{\bf [RD.1]}]
    \item $R[f+g] = R[f] + R[g]$ and $R[0] = 0$;
    \item for all $a,b,c$: $\<a,b+c\>R[f] = \<a,b\>R[f] + \<a,c\>R[f]$ and $\<a,0\>R[f] = \rs{af}0$;
    \item $R[\pi_j] = \pi_1 \iota_j$;
    \item $R[\<f,g\>] = (1\x \pi_0)R[f] + (1\x \pi_1)R[g]$;
    \item $R[fg] = \<\pi_0,\<\pi_0f,\pi_1\>R[g]\>R[f]$;
    \item $\<1\x\pi_0,0\x \pi_1\>(\iota_0\x 1)R[R[R[f]]]\pi_1 = (1\x \pi_1)R[f]$;
    \item $(\iota_0\x 1)R[R[(\iota_0 \x 1)R[R[f]]\pi_1]]\pi_1 = \ex (\iota_0\x 1)R[R[(\iota_0 \x 1)R[R[f]]\pi_1]]\pi_1$;
    \item $\rs{R[f]} = \rs{f} \x 1$;
    \item $R[\rs{f}] = (\rs{f} \x 1) \pi_1$.
  \end{enumerate}
\end{definition}

As noted above, \textbf{[RD.5]} represents the chain rule, while \textbf{[RD.8]} says that the partiality of $R[f]$ is entirely determined by the partiality of $f$ itself.  \textbf{[RD.9]} says how to differentiate restriction idempotents.  The other axioms are similar to those for an RDC; for an explanation of what they represent, see the discussion after Definition 13 in \cite{arxiv:RDC}. 
Also note that term logics have also been given to simplify reasoning in cartesian differential categories \cite{journal:BCS:CDC} and 
differential restriction categories \cite{msc:gallagher-calgary}; a term logic for reverse differential restriction categories exists 
but will not be discussed further here.

Any Fermat theory \cite{paper:On-1-Form-Classifiers} and more generally any Lawvere theory which is also 
a cartesian differential category can be given the structure of a reverse differential category; in these 
cases both the forward and reverse derivatives can be pushed down to sums and tuples of derivatives on 
maps $R \to R$, and here the forward and reverse derivative necessarily coincide.  A restriction version of this 
example is given by considering a topological ring $R$ that satisfies the axiom of determinacy 
(see \cite{journal:bertram-calc-arbring}); the category with objects: powers of $R$, and $C^\infty$-maps that are smooth 
on restriction to an open set form a reverse differential restriction category.  This meta-example includes 
the category $\smooth_P$ of functions that are smooth on an open subset of $\R^n$.  For an example whose 
objects are not of the form $R^n$: the coKleisli category of the multiset comonad on the category of 
relations $\mathsf{Rel}$ is a cartesian differential category and the category of linear maps 
is $\mathsf{Rel}$ (see \cite{DiffCats,journal:BCS:CDC} for details).
As $\mathsf{Rel}$ is a compact closed self-dual category,
and the derivative at a point is linear (hence a map in $\mathsf{Rel}$), one can obtain a reverse derivative on the coKleisli category of the 
finite multiset comonad on $\mathsf{Rel}$.

Just as with an RDC, we  can derive a forward differential restriction structure from a reverse.

\begin{theorem}\label{theorem:reverse-are-forward}
  Every reverse differential restriction category $\X$ is a differential restriction category 
  with the derivative defined as previously (see also \cite[Theorem 16] {arxiv:RDC}).  
\end{theorem}

Moreover, just as in \cite[Theorem 42]{arxiv:RDC}, one can prove that a DRC with a ``contextual linear dagger'' is equivalent to an RDRC; however, for space constraints we will not go into full details here.  
One must first describe fibrations for restriction categories: these were studied by Nester in \cite{msc:nester-calgary}.  One can give 
a version of the simple fibration for a restriction category as well as the dual of the simple fibration (this is 
remarkable -- as the dual of a restriction category is not generally a restriction category).  Importantly, maps in the simple fibration 
have their partiality concentrated in the context i.e. $\rs{f} = e \x 1$ where $e=\rs{e}$.   A contextual dagger is a 
an involution of fibrations $\lin(\X)[\X] \to \lin(\X)[\X]^*$ where $\lin(\X)[\X]$ denotes a subfibration of the simple fibration consisting of linear maps in context, using the notion of fibration for restriction categories.
From a reverse differential restriction category one obtains such an involution of fibrations from $(u,f) \mapsto (u,(\iota_0 \x 1)R[f]\pi_1)$,
and the second component is sometimes written $f^{\dagger[I]}$ where $I$ is the context object.  There are a few 
subtleties that we will also not go further into here.

The reverse derivative automatically preserves the induced partial order (from the restriction structure) and joins, if they exist:

\begin{proposition}\label{proposition:rev-diff-join-pres}
If $\X$ is a reverse differential restriction category, then for any $f, g: A \to B$, $f \leq g$ implies $R[f] \leq R[g]$, and if $\X$ has joins, then for any pairwise 
  compatible family $\{f_i\}$, $R[\vee_i f_i] = \vee_i R[f_i]$.
\end{proposition}

As we shall see, we will not strictly need {\bf[RD.6]} and {\bf[RD.7]} to model Abadi and Plotkin's language; thus, we make the following definition:
 
\begin{definition}
A \textbf{basic reverse differential restriction category} (or \textbf{basic RDRC}) is a structure satisfying all the requirements for an RDRC except {\bf[RD.6]} and {\bf[RD.7]}.
\end{definition}

However, as we discuss in the final section, using axioms {\bf[RD.6]} and {\bf[RD.7]} allows one to consider improvements to the operational semantics of the language.

\section{Interpretation of a simple differential language}\label{sec:interpretation}

We will make use of the language defined by Abadi and Plotkin \cite{journal:Abadi-Plotkin}.
We will make one modification up front.  We will first consider the language without recursive function 
definitions and instead with while-loops (called SDPL); after showing the semantics works out, we will then add 
recursive definitions back in (called SDPL$^+$).  We remark that while the presentation of $\sdpl$ given in 
Plotkin and followed here is parametrized over a single generating type; however, we can add arbitrary 
generating types as long as those types have operations that provide the structure of a commutative monoid.

In \cite{journal:Abadi-Plotkin}, Abadi and Plotkin remarked that there are two approaches to 
differentiating over control structures: there are source transformations used in systems such as 
TensorFlow \cite{journal:tensorflow-ref} and Theano \cite{journal:theano-ref} and there is the execution trace method used in systems 
such as Autograd \cite{journal:autograd-ref} and PyTorch \cite{journal:pytorch-ref}.  The source transformation method for dealing with 
derivatives of control structures defines a way to distribute the derivative into control structures; for example 
  \[\difft{\,\texttt{if} \, B \, \texttt{then} \, M \, \texttt{else} \, N}{x}\]
would be replaced by 
  \[\texttt{if} \, B \, \texttt{then} \, \difft{m}{x} \, \texttt{else} \, \difft{n}{x}\]

The execution trace allows defining a symbolic derivative 
on simpler terms with no control structures or derivatives, and then evaluating 
a term enough so that there are no control structures or derivatives present, allowing a symbolic 
trace through the derivative.  This must be done at runtime -- for example, we need to know when differentiating
over an \text{if-then-else} statement which branch was taken, and once this control structure is eliminated 
the derivative can be computed on the simpler resultant term.  This has the advantage of making it simpler to adapt to 
derivatives over more subtle structures such as recursive function definitions.  Since it's done at runtime,
it can performed by a source-transformation by a JIT compiler ensuring efficiency.

\subsection{The core language SDPL}
The types of $\sdpl$ are given by the following grammar:
  \[
    \ty  := \real \, | \, 1 \, | \, \ty \x \ty
  \]
Powers are assumed to be left-associated so $\real^{n+1} := \real^{n} \x \real$.  
To form the raw terms of $\sdpl$ we assume a countable supply of variables, a set of 
typed operation symbols $\Sigma$, and a set of typed predicate symbols $\mathsf{Pred}$.
The raw terms are then defined by the following grammar:
  \begin{align*}
    m := 
    & \, x \, 
    | \, r \ (r \in \R) \,
    | \, m + m \,
    | \, \op(m) \ (\op \in \Sigma) \,
    | \, \texttt{let} \, x : \ty = m \, \texttt{in}\, m  \\
    & | \,  * \, 
    | \, (m,n)_{\ty,\ty} \, 
    | \, \fst_{\ty,\ty}(m) \,
    | \, \snd_{\ty,\ty}(m) \,
    | \, \texttt{if} \, b \, \texttt{then} \, m \, \texttt{else} \, n  \\
    & | \, \texttt{while} \, b \, \texttt{do} \, m \, 
    | \,m.\rd(x:T.m)(m) \\ ~\\
    b := & \mathsf{pred}(m) \ (\mathsf{pred} \in \mathsf{Pred}) \, | \, \true \, | \, \false
  \end{align*}

Note that the typing rules will disallow inputs or outputs to come from boolean terms.  This is to ensure that 
all typed terms are differentiable with respect to every argument.  The typing rules for $\sdpl$ are
given in Table \ref{table:sdpl-typing}.
\begin{table} 
  \begin{center}
  \fbox{ 
    \begin{tabular}{lll}
      \infer{\Gamma,x:A \proves x: A}{}
        & \infer{\Gamma \proves r : \real}{r \in \R} 
        & \infer{\Gamma \proves m + n : \real}{\Gamma \proves m : \real & \Gamma \proves n : \real}\\~\\
      \infer{\Gamma \proves \op(m) : U}{\Gamma \proves m : T & \op : T \to U \in \Sigma}
        & \multicolumn{2}{c}{
          \infer{\Gamma \proves \texttt{let}\, x:T = m \, \texttt{in} \, n : U}{\Gamma \proves m : T & \Gamma,x:T \proves n : U}
        }\\~\\
      \infer{\Gamma \proves * : 1}{}
        \quad \infer{\Gamma \proves (m,n)_{U,T}: U \x T}{\Gamma \proves m : U & \Gamma \proves n : T}
        & \infer{\Gamma \proves \fst_{U,T}(m):U}{\Gamma \proves m: U\x T}
        & \infer{\Gamma \proves \snd_{U,T}(m):T}{\Gamma \proves m: U\x T}\\~\\
      \infer{\Gamma \proves \texttt{if} \, b \, \texttt{then} \, m \, \texttt{else}\, n : T}{\Gamma \proves b & \Gamma \proves m : T & \Gamma \proves n : T}
        & \multicolumn{2}{c}{
            \infer{p:U \proves \texttt{while} \, b \, \texttt{do} \, f : U}{p:U \proves b & p:U \proves f:U}
          }\\~\\
        \multicolumn{3}{c}{
          \infer{\Gamma \proves v.\rd(x:U.m)(a):U}{\Gamma,x:U \proves m : T & \Gamma \proves a : U & \Gamma \proves v:T}
        }\\~\\
      \infer{\Gamma\proves \true}{}
        & \infer{\Gamma \proves \false}{}
        & \infer{\Gamma \proves \mathsf{pred}(m)}{\Gamma \proves m : U & \mathsf{pred} : U \in \mathsf{Pred}}
    \end{tabular}
  }
  \end{center}
  \caption{Typing rules for $\sdpl$}
  \label{table:sdpl-typing}
\end{table}
In the typing rules, $\Gamma$ is assumed to be a list of typed variables $\Gamma = [x_i : A_i]_{i=1}^n$ where $A_i \in \ty$.  
Free variables are defined in the usual way; note that $\texttt{let}$ expressions bind the variable $x$ and 
when forming the reverse differential term $v.\rd(x:U.m)(a)$ the variable $x$ is also bound.  The reverse differential 
expression may be read as ``the reverse differential of $m$ with respect to $x$ evaluated at the point $a$ in the 
direction $v$.''

\subsection{Categorical interpretation of $\sdpl$}\label{section:cat-interp-sdpl}

Let $\X$ be a basic reverse differential restriction category with countable joins of disjoint maps.  
An \textbf{interpretation structure} for $\sdpl$ into $\X$ is given by a tuple of structures:
\[
  (A \in \X_0,(1\to^{a_r}A)_{r\in \R},\den{\blank},\den{\blank}_T,\den{\blank}_F)  
\]
and we extend such a structure to an interpretation  of all the terms of $\sdpl$ as explained below.
We must first interpret types, and to begin we need an object $A$ from $\X$ to carry our signatures.  
We also require that $A$ has a point $1 \to^{a_r}A$ for each element $r\in \R$ (since we must interpret $\R$ 
constants which are part of $\sdpl$) \footnote{It is not strictly necessary that $\sdpl$ contains a constant 
for every $r\in \R$ -- as long as we include $0$
 we could only require constants that we actually use, such as the computable reals.}.  
With such an $A$ we define an interpretation of types:
\[
  \den{1} := 1 \qquad \den{\real} := A \qquad \den{T\x U} := \den{T} \x \den{U}  
\]
We extend the interpretation to contexts:
\[
 \den{\cdot} := 1\qquad \den{x:U} := \den{U} \qquad   \den{\Gamma,x:U} := \den{\Gamma} \x \den{U}
\]
We also require an interpretation of each operation symbol $\op : T \to U \in \Sigma$ of the correct type:
$\den{\op} : \den{T} \to \den{U}$.  We additionally require two interpretations of each predicate symbol 
$\mathsf{pred} : U \in \mathsf{Pred}$: $\den{\mathsf{pred}}_T : \den{U} \to 1$ and $\den{\mathsf{pred}}_F : \den{U} \to 1$
such that $\rs{\den{\mathsf{pred}}_T} \rs{\den{\mathsf{pred}}_F} = \emptyset$.  To summarize: 
\[
  \Sigma(U,T) \to^{\den{\blank}} \X(\den{U},\den{T}) \qquad 
  \mathsf{Pred}(U) \to^{\den{\blank}_T,\den{\blank}_F} \X(\den{U},1)
\]

The intent for giving two interpretations of predicate symbols is that we must give an interpretation 
of the ``true'' part of the predicate and the ``false'' part.  In \cite{journal:Abadi-Plotkin}  an interpretation of predicate symbols is given as maps $\den{U} \to \{\true,\false\}$ with the property 
that the preimages of both $\true$ and $\false$ are open.  This necessarily makes the interpretation of 
a predicate partial or trivial, and moreover it is equivalent to giving an interpretation 
of predicate symbols into disjoint open sets of $\den{U}$, which is again equivalent to giving 
an interpretation into disjoint predicates on $\den{U}$.  A way around this non-standard 
interpretation of predicates is given by taking the 
manifold completion \cite{journal:grandis-manifolds,journal:TangentCats} of the model, noting 
that $1+1$ is a manifold, and then requiring that we map into $1+1$ by an atlas morphism, which will necessarily
yield two disjoint restriction idempotents on the domain.  Another approach is to use the Heyting negation of the 
associated restriction idempotent, noting that this will always be disjoint from the starting map.  
These two approaches have interesting relationships with the approach we take, but their full development 
wil not be pursued in this work.

We then extend the interpretation to all terms inductively.  Most of these interpretations are standard; the more novel parts are the interpretations of if, while, and reverse derivatives.  
\begin{description}
  \item[Proj:]
    ~
    \begin{itemize}
      \item $\den{x:U \proves x:U} := 1_{\den{U}}$;
      \item $\den{\Gamma,x:U \proves x:U} := \den{\gamma} \x \den{U} \to^{\pi_1} \den{U}$;
      \item $\den{\Gamma,y:U \proves x:T} := \den{\gamma} \x \den{U} \to^{\pi_0} \den{\gamma} \to^{\den{\Gamma \proves x:T}} \den{T}$.
    \end{itemize} 
  \item[Real operations:] ~
    \begin{itemize}
      \item We define $\den{\Gamma \proves 0:\real} := \den{\Gamma} \to^{0} A$ and for the other elements
      $
        \den{\Gamma \proves r : \real} := \den{\Gamma} \to^{!} 1 \to^{a_r = \den{r}} A = \den{\real}
      $
      \item 
      $
        \den{\Gamma \proves m+n : \real} := \den{\Gamma} \to^{\den{\Gamma \proves m:\real} + \den{\Gamma\proves n : \real}} \den{\real}
      $
    \end{itemize}
  \item[Operation terms:] Given $\op : T \to U \in \Sigma$ 
    \[
      \den{\Gamma \proves \op(m):U} := \den{\Gamma} \to^{\den{\Gamma\proves m: T}} \den{T} \to^{\den{\op}} \den{U}
    \]
  \item[Let:] 
    \[\den{\Gamma \proves \texttt{let} \, x:T = m \, \texttt{in}\, n : U}
        := \den{\Gamma} \to^{\<1,\den{\Gamma\proves m:T}\>} \den{\Gamma} \x \den{T} \to^{\den{\Gamma,x:T \proves n :U}} \den{U}\]
  \item[Product terms:]~
    \begin{itemize}
      \item $\den{\Gamma \proves * : 1} := \den{\Gamma} \to^{!} 1$
      \item $\den{\Gamma \proves (m,n)_{A,B} : A\x B} := \den{\Gamma} \to^{\<\den{\Gamma\proves m:A},\den{\Gamma\proves n:B}\>} \den{A}\x \den{B}$
      \item $\den{\Gamma \proves \fst(m)_{A,B}:A} := \den{\Gamma} \to^{\den{\Gamma\proves m:A\x B}} \den{A} \x \den{B} \to^{\pi_0}\den{A}$
      \item $\den{\Gamma \proves \snd(m)_{A,B}:B} := \den{\Gamma} \to^{\den{\Gamma\proves m:A\x B}} \den{A} \x \den{B} \to^{\pi_1}\den{B}$
    \end{itemize}
  \item[Control structures:] ~
      \[\den{\Gamma \proves \texttt{if}\, b\, \texttt{then}\, m\, \texttt{else}\, n :U}
        := \rs{\den{\Gamma\proves b}_T}\den{\Gamma \proves m: U} \vee \rs{\den{\Gamma\proves b}_F}\den{\Gamma\proves n :U}\]
      \[
        \den{p:A \proves \texttt{while}\, b\, \texttt{do}\, m :A}
        := \bigvee_{i=0}^{\infty}\left( \left(\rs{\den{p:A\proves b}_T}\den{p:A\proves m:A}\right)^i \rs{\den{q:A\proves b}_F} \right)
      \]
  \item[Reverse derivatives:] ~
    \begin{align*}
      & \den{\Gamma \proves v.\rd(x:T.m)(a):T} \\ &:=
        \den{\Gamma} \to^{\<\<1,\den{\Gamma\proves a:T}\>,\den{\Gamma\proves v:U}\>}
        (\den{\Gamma} \x \den{T})\x \den{U}
          \to^{R[\den{\Gamma,x:T\proves m:U}]\pi_1}\den{T}
    \end{align*}
  \item[Boolean terms:] $\den{\Gamma \proves \true}_T := \den{\Gamma} \to^{!} 1$ and $\den{\Gamma\proves \true}_F := \emptyset$.  
                        Likewise, $\den{\Gamma\proves \false}_T := \emptyset$ and \\ $\den{\Gamma \proves \false}_F := !$.    Finally 
                        for any $\mathsf{pred}\in \mathsf{Pred}(A)$:
    \[
      \den{\Gamma \proves \mathsf{pred}(m)}_H := \den{\Gamma} \to^{\den{m}}\den{A} \to^{\den{\mathsf{pred}}_H} 1
    \] 
    where $H$ ranges over $\{T,F\}$.
\end{description}

For a brief explanation of the interpretation of while-loops, for $f : A \to A$ we set $f^0 = \mathsf{id}$ and 
$f^{n+1} = ff^n$.  Then our interpretation says either the guard was false, or it was true and we executed $m$ and
then it was false, or it was true and we executed $m$ and it was still true and we executed $m$ again and then it was 
false, and so on.  This yields
  \[
    \den{\texttt{while}\, b \, \texttt{do}\, m} 
    := \den{b}_F \vee \den{b}_T \den{m}\den{b}_F \vee \den{b}_T \den{m} \den{b}_T \den{m}_T\den{b}_F \vee \cdots
  \]

\subsection{Categorical semantics of source code transformations}\label{section:source-transformation}
In this section we show that the interpretation above always soundly models source code transformations 
for differentiating if-then-else statements and while-loops.

\begin{proposition}
  In an interpretation structure on a basic RDRC, 
  for any terms $\Gamma,x:U \proves m:T$, $\Gamma,x:U\proves n:T$, $\Gamma\proves a:U$, and $\Gamma\proves v:T$ and 
  for any predicate $\Gamma,x:U \proves B$ we have 
  \begin{align*}
    &\den{\Gamma\proves v.\rd(x:U. \texttt{if}\, b \, \texttt{then}\, m \, \texttt{else} \, n  )(a)}  \\
    &=
    \den{\Gamma \proves \texttt{if} \, (\texttt{let}\, x = a\, \texttt{in}\, b) \, \texttt{then}\, v.\rd(x:U.m)(a)\, \texttt{else} \, v.\rd(x:U.n)(a)}
  \end{align*}
\end{proposition}

\begin{corollary}[If-then-else transformation]\
  In an interpretation structure on a basic RDRC, we always have 
    \begin{align*}
      & \den{\Gamma,x:U \proves v.\rd(x:U. \texttt{if}\, b \, \texttt{then}\, m \, \texttt{else}\, n)(x)}\\
      &= \den{\Gamma,x:Y \proves \texttt{if}\, b\, \texttt{then} \, v.\rd(x:U.m)(x)\, \texttt{else}\, v.\rd(x:U.n)(x)}
    \end{align*}  
\end{corollary}

Turning to iteration, if a while-loop terminates, then 
$\texttt{while}\, b\, \texttt{do}\, f$ is $f^n$ for some $n$.  The forward derivative 
admits a tail recursive description:
  \[
    D[f^{n+1}] = \<\pi_0f,D[f]\>D[f^n]  
  \]

$\sdpl$ has two admissable operations: dagger and forward differentiation.

\[
  m^{\dagger[\Gamma]} := y.\rd(x.m)(0) \qquad \fd(x.m)(a).v := \texttt{let}\, z = v \, \texttt{in} \, (y.\rd(x.m)(a))^{\dagger[\Gamma]}
\]
where the $y$ in $m^{\dagger[\Gamma]}$ is fresh.  The recursive description of $D[f^n]$ is useful in proving the following:

\begin{proposition}[Forward-differentiation for while-loops]\label{proposition:forward-diff-while}
  In an interpretation structure on a basic RDRC,
  \begin{enumerate}
    \item For any $\Gamma,x:A\proves m:B$,
      $
      \den{\fd(x.m)(a).v}
      =
       \<\<1,\den{a}\>,\den{v}\> (1\x \iota_1)D[\den{m}]$\
    \item For any $x:A \proves f:A$ we have
        $
        \den{\proves \fd(x.\texttt{while}\, b \, \texttt{do} \, f)(a).v}\\
        = 
        \den{\proves\texttt{let}\, x=a,y=v\, \texttt{in} \, 
              \snd(
                \texttt{while}\, \pi_0b \, \texttt{do}\,
                  (\pi_0f,\fd(x.f)(x).y)
              )}
              $
  \end{enumerate}
\end{proposition}

On the other hand, the reverse derivative satisfies:
  \[
    R[f^n](a,b) = R[f](a,R[f](f(a),R[f](f(f(a)),\cdots ,b)))  
  \]
Which looks at first glance to be head recursive, and not like something that can be implemented by an 
iteration.  However, with {\bf [RD.6]}, we can do the following:
\[
  R[f^{n+1}] = D[f^{n+1}]^{\dagger[A]} = (T(f)^nD[f])^{\dagger[A]}  \qquad \text{where } T(f) = \<\pi_0f,D[f]\>
\]
This is the basis of the following source transformation for while-loops.

\begin{corollary}[Reverse-differentiation of while-loops]
   In an interpretation structure on an RDRC, let $z:A \proves f:A$ and $z:A \proves b$; then 
   we have 
   \begin{align*}
     & \den{v:A \proves  v.\rd(x.\texttt{while}\, b \, \texttt{do}\, f)(a)}\\
     &=
     \den{v:A \proves (\proves\texttt{let}\, x=a,y=v\, \texttt{in} \, 
     \snd(
       \texttt{while}\, \pi_0b \, \texttt{do}\,
         (\pi_0f,\fd(x.f)(x).y)
     ))^{\dagger[.]}}
\end{align*}
   where $\dagger[.]$ denotes the dagger defined above with respect to the empty context.
\end{corollary}

\subsection{Smooth recursive definitions}
Until now we discussed the semantics of a fragment of the language described by \cite{journal:Abadi-Plotkin}.
We formally extended their language with while-loops to isolate their behaviour, but missed out on recursive function definitions.  
Given general recursion, one can implement loops using tail recursion.  We now move to discuss their 
full language with recursive function definitions.  This language will be 
called $\sdpl^+$.  To give such an extension, we 
introduce two new raw terms 

\[
  m \, := \, m \ \text{as before} \, | \, f(m) \, | \, \texttt{letrec}\, f(x) \, := m \, \texttt{in} \, n  
\]

In the above, when we form $f(a)$, the symbol $f$ is taken to be a free function variable, and the 
term $\texttt{letrec}\, f(x) \, := \, m \, \texttt{in}\, n$ binds the variable $x$ in $m$ and the 
function variable $f$ in $m$ and $n$.  However, these function variables are of a 
different sort than ordinary variables because they have arity.  That is $f(a)$ only makes sense if $a:B$ and 
$f$ has arity $B\to C$, which we write as $f:B\to C$.  Thus, our typing/term formation rules will have two 
sorts of contexts, one to record function names and the other for ordinary variables.  Our terms in context then 
have the form $\Phi | \Gamma \proves m : B$, and to update the rules from before, just add $\Phi$ to all the contexts.  
The two new rules are 
\[
 \infer{\Phi,f:A\to B|\Gamma \proves f(m):B}{\Phi,f:A \to B|\Gamma \proves m:A}  
 \qquad 
 \infer{\Phi|\Gamma \proves \texttt{letrec}\, f(x)\, :=\, m\,\texttt{in} \, n :C}{\Phi,f:A\to B|x:A \proves m & \Phi,f:A \to B|\Gamma \proves n : C}
\]

We will now give the interpretation of recursive definitions and calls in a basic reverse differential 
join restriction category.  But first, we will review a basic bit of intuition of recursive function theory 
in case the reader is unfamiliar.  

We often write computable functions $A \to^{f} B$ as $f(n) := m$, 
but it is usually helpful to think of $f$ as simply a name for the unnamed function $\lambda n.m$ and 
then write $f = \lambda n. m$.  The idea is that as a computation $f$ has an internal representation 
that uses the variable $n$ somewhere. If $f$
is recursive then that means that the symbol $f$ also appears in $m$, and thus it is a function that depends 
on itself.  To break this cycle we then abstract out the symbol $f$ too.   We write 
$\underline{f} := \lambda f . \lambda n . m$.  This creates a function 
\[
  \fun(A,B) \to^{\underline{f}} \fun(A,B)
\]
$\underline{f}$ takes an arbitrary computable function $A \to^{h} B$ and creates a function that uses $h$ 
instead of $f$ anywhere $f$ was used in the body $m$.  

To give a quick example, consider the computable function 
$\mathsf{fac}(n) := \texttt{if} \, n < 1 \, \texttt{then}\, 1 \, \texttt{else} \, n*\mathsf{fac}(n-1)$.
Then 
\[
  \underline{\mathsf{fac}}(h)(n) := \texttt{if} \, n < 1\, \texttt{then}\, 1\, \texttt{else}\, n*h(n-1)  
\]
The point is that this new function is not recursive.  However, it is instructive to see what happens 
when we apply it to the function it represents.  As an exercise, we leave it to the reader to prove that 
\[
    \underline{\mathsf{fac}}(\mathsf{fac}) = \mathsf{fac}
\]
In other words, the recursive function $\mathsf{fac}$ is a \emph{fixed point} of the functional $\underline{\mathsf{fac}}$.
It is also the best fixed point of $\mathsf{fac}$ in the sense that it is the least defined function 
that is a fixed point of $\underline{\mathsf{fac}}$.  This works in general, given any recursive function 
$r$ it may be obtained as the least fixed point of $\underline{r}$.

To model least-fixed-point phenomena we will use the notion of a pointed directed complete partial order or (DCPPO) for
short.  The first use of DCPPOs to model recursive phenomena is due to Scott \cite{journal:scott-domains-first} in 
giving models of the untyped $\lambda$-calculus.  DCPPOs are used in the semantics of the functional 
programming language PCF in \cite{journal:pcf-original}.  Abstract DCPPO-enriched categories of partial 
maps were used in modelling the semantics of the functional programming language FPC in \cite{journal:pmaps-dcppo-fpcs}.
The DCPPO structure on homsets of $\smooth_P$ was used in \cite{journal:Abadi-Plotkin} to provide a semantics 
of SDPL.  The approach taken here generalizes \cite{journal:Abadi-Plotkin} to an arbitrary 
basic reverse differential join restriction category, highlights the structural aspects of the interpretation,
and uses the axioms of such a category to derive some simplifications to the operational behaviour.  
A connection of $\omega$-CPPOs and restriction categories was introduced using the delay monad in 
\cite{chapter:delay-restriction}.

\begin{definition}
Let $(D,\leq)$ be a partial order.  A subset $A \subseteq D$ is {\bf directed} if $A$ is nonempty 
and any two elements $f,g \in A$ have an upper bound in $A$; i.e., there is an $h \in A$ with $f,g\leq h$.
A partial order $(D,\leq)$ is a {\bf directed complete partial order} if every directed subset $A$ has a 
supremum written $\bigvee_{a\in A} a \in D$.  A directed complete partial order is {\bf pointed} (DCPPO) if 
there is a supremum for the empty set, that is a minimal element $\emptyset \leq d$ for all $d \in D$.  

By a {\bf morphism of DCPPOs} $(P,\leq) \to^{g} (Q,\leq)$ we mean a function $g$ on the underlying sets that 
is monotone and preserves suprema.  We observe minimally that the category of DCPPOs is Cartesian closed.
\end{definition}

\begin{lemma}\cite{journal:Kleene-FP-Theorem}\label{lemma:dcppo-fixpoint-lemma}
  Let $(D,\leq)$ be a DCPPO.  Then 
  \begin{enumerate}
    \item Every morphism $D \to^{g} D$ has a least fixed point; i.e. a $u\in D$ such that $g(u) = u$.
    \item For any other DCPPO $(P,\leq)$ every morphism $P \x D \to^{g} P$ has a parametrized fixed point; i.e., 
          a $P \to^{u} D$ such that 
          \[
            \begin{tikzcd}
              P \ar[r,"u"] \ar[dr,"{\<1,u\>}"'] & D \\
              & P \x D\ar[u,"g"']
            \end{tikzcd}  
          \] 
          In other words, for each $x \in P$, $u(x)$ is a fixed point of $g(x,\blank)$. 
          This parametrized fixed point is often denoted $\mu_y.g(\blank,y)$ and as the 
          fixed point of $g(x,\blank)$ by $\mu_y.g(x,y)$.
  \end{enumerate}
\end{lemma}

Join restriction categories are DCPPO enriched:
\begin{proposition}\label{proposition:join-implies-dcppo}
  Let $\X$ be a restriction category.  Then with respect to the order enrichment of restriction categories: 
  \begin{enumerate}
    \item $\X$ is a join restriction category then the enrichment lies in DCPPOs.
    \item If $\X$ has joins and restriction products then those products are DCPPO enriched products; i.e., $\X(A,B\x C) \simeq \X(A,B) \x \X(A,C)$ qua an isomorphism of DCPPOs.
          Moreover, the ``contraction operator''
          \[
            \X(A,B) \to^{\Delta_A}   \X(A,A\x B)
          \]
          that sends $A \to^{f} B$ to $A \to^{\<1,f\>} A\x B$ is a morphism of DCPPOs.
    \item If $\X$ has joins and is a Cartesian left additive restriction category, then the addition on homsets
          \[
            \X(A,B) \x \X(A,B) \to^{+} \X(A,B)  
          \]
     is a morphism of DCPPOs.
    \item If $\X$ is a reverse differential join restriction category then the operation 
      of reverse differentiation
      \[
        \X(A,B) \to^{R[\blank]}   \X(A\x B,A)
      \]
      is a morphism of DCPPOs.
  \end{enumerate}
\end{proposition}

Parts \emph{2}--\emph{4} of Proposition \ref{proposition:join-implies-dcppo} implies that certain operations we will 
need to form from monotone and join preserving maps will again be monotone and join preserving.  

To give the categorical semantics of $\sdpl^+$, we must extend the interpretation developed in section \ref{section:cat-interp-sdpl}.
To begin we first give the interpretation of function contexts.  The idea being that a free function symbol 
could be any map of the correct type, the interpretation of function contexts is given as a product of homsets.
\[
  \den{\emptyset}:= 1 \qquad \den{\Phi,f:A\to B} := \den{\Phi}\x \X(\den{A},\den{B})  
\]
The interpretation of a term in context $\Gamma \proves m : B$ constructed a map $\den{\Gamma} \to^{\den{m}} \den{B}$.
With function contexts, the maps we build now depend on the morphism from $\X(\den{A},\den{B})$ to fill in 
the call to a function.  That is, the interpretation is now a function
\[
    \den{\Phi} \to^{\den{\Phi|\Gamma \proves m : B}} \X(\den{\Gamma},\den{B})
\]

Now, we are building a function and to do so it suffices to build a map in $\X(\den{\Gamma},\den{B})$ for 
each element $\phi \in \den{\Phi}$.  We write $\den{m}_\phi$ for the value of $\den{m}$ at $\phi$.  The 
construction is by induction and for the terms from $\sdpl$, the construction is exactly the same 
with the addition of a $\phi$ subscript decorating the terms appropriately.  
For example, $\den{\Phi|\Gamma\proves \texttt{let}\, x=m\, \texttt{in}\, n}_\phi := \<1,\den{m}_\phi\>\den{n}_\phi$.
However, we can build the interpretation entirely using external structure by induction as well.

For example, 
interpretation of $\texttt{let}\, x=m\, \texttt{in}\, n$ may be given using the ``contraction operator'', and this construction is element free. 
\begin{tiny}
\[
\begin{tikzcd}[column sep=4em]
  \den{\Phi} \ar[r,"{\<\den{m},\den{n}\>}"] \ar[drr,"\den{\texttt{let}\, x=m\, \texttt{in} \, n}"']
  & {\X(\den{\Gamma},\den{A})\x \X(\den{\Gamma}\x \den{A},\den{B})} \ar[r,"\Delta_\Gamma\x 1"]
  & {\X(\den{\Gamma},\den{\Gamma}\x \den{A})\x \X(\den{\Gamma}\x \den{A},\den{B})} \ar[d,"\cdot"]\\
  && {\X(\den{\Gamma},\den{B})}
\end{tikzcd}
\]
\end{tiny}
We will leave it to the reader to construct the interpretation of $v.\rd(x.m)(a)$ using a similar 
idea, as well as the reverse differential operator $\X(A,B) \to^{R[\blank]} \X(A\x B,A)$.

For $\sdpl^+$, we extend this to the two new terms.  Given a function context $\Phi = (f_1,\ldots,f_n)$ 
then for any $\phi\in\den{\Phi}$ we have that $\phi=(\phi_1,\ldots,\phi_n)$ are all maps in $\X$: 
if $f_i : A_i \to B_i$ then $\phi_i : \den{A_i} \to \den{B_i}$.  We will write $\phi(f_i)$ to denote $\phi_i$.
We also make use of the ``no-free-variable'' assumption for recursive definitions; that is, in the type
formation rule for recursive definitions $\texttt{letrec}\, f(x):=m\, \texttt{in} \, n$, $m$ must 
have at most a unique free variable, and it must be $x$.

\begin{description}
  \item[Fun-Call:] $\den{\Phi,f:A \to B|\Gamma\proves f(m):B}_\phi := \den{\Gamma} \to^{\den{m}_\phi} \den{A} \to^{\phi(f)} \den{B}$
  \item[Rec-Def:] First note that if we just translate a simple recursive function $\texttt{letrec} \, f(x) := m$,  
                  we see that $x$ is a free variable and $f$ is a free function variable in $m$.  
                  That is, we have $f:A\to B|x:A \proves m:A$.  Then note that the 
                  interpretation we are developing would interpret $m$ as a function 
                  \[
                    \X(\den{A},\den{B}) \to^{\den{m}} \X(\den{A},\den{B})  
                  \]
                  This is exactly the sort of underlined function we looked at earlier: it takes each $h: \den{A}\to\den{B}$ in $\X$
                  and uses it by the above translation of function calls, any where that $f$ was used in $m$.  
                  Then by Lemma \ref{lemma:dcppo-fixpoint-lemma}, we may take its fixed point, $\mu$.  We 
                  then get a map $\den{A} \to^{\mu} \den{B}$ such that $\den{m}\mu = \mu$ and is the least 
                  defined such map, giving us the interpretation of the recursive function, 
                  and we would write $\den{\mathsf{letrec}\, f(x):=m} = \mu$.  More generally, in $m$ the 
                  unique variable condition only applies to ordinary variables, but $m$ could have multiple function 
                  variables.  Then if we translate $\Phi,f:A\to B|x:A \proves m:B$ we get a map 
                  \[
                    \den{\Phi}\x \X(\den{A},\den{B}) \to^{\den{m}} \X(\den{A},\den{B})  
                  \]
                  We may then apply the second part of Lemma \ref{lemma:dcppo-fixpoint-lemma} and obtain a
                  parametrized fixed point 
                  \[
                   \den{\Phi} \to^{\mu_f.\den{m}_{(\blank,f)}}  \X(\den{A},\den{B})
                  \]
                  Likewise if we translate $\Phi,f:A\to B|\Gamma \proves n : C$, we get a map 
                  \[
                   \den{\Phi}\x \X(\den{A},\den{B}) \to^{\den{n}} \X(\den{\Gamma},\den{C})  
                  \]
                  Then, finally, the interpretation of $\texttt{letrec}\, f(x):=m\, \texttt{in}\, n$ is 
                  defined by the following diagram:
                  \[
                    \begin{tikzcd}[column sep=6em]
                      \den{\Phi}
                        \ar[r,"{\<1,\mu_f.\den{m}_{(\blank,f)}\>}"] 
                        \ar[dr,"\den{\texttt{letrec}\, f(x):=m\,\texttt{in}\, n}"']
                         & \den{\Phi} \x {\X(\den{A},\den{B})} \ar[d,"\den{n}"]\\
                      & \X(\den{\Gamma},\den{C})
                    \end{tikzcd}
                  \]
                  We may also define it componentwise as 
                  \[
                   \den{\texttt{letrec}\, f(x):=m\, \texttt{in}\, n}_\phi
                    := \den{n}_{(\phi,\mu_f.\den{m}_{(\phi,f)})} 
                  \]
\end{description}

Note that the above definition is only well-defined if we can prove that the interpretation 
$\den{\Phi}\to^{\den{m}} \X(\den{\Gamma},\den{B})$ always yields a monotone and join preserving 
function between the DCPPOs, so that in the last step the use of Lemma \ref{lemma:dcppo-fixpoint-lemma} 
is justified.

\begin{proposition}\label{proposition:recursive-defn-well-defined}
  Let $\X$ be a basic reverse differential join restriction category, with a specified interpretation structure
  for $\sdpl^+$.  Then the interpretation of terms in context is always a monotone, join preserving 
  function between the DCPPOs.  In particular, the construction is well-defined.
\end{proposition}

\subsection{Operational semantics}
The operational semantics used by \cite{journal:Abadi-Plotkin} defined a sublanguage of the raw terms called 
\textsf{trace terms}.  These are generated by the following grammar:

\[
  \mathsf{tr} := x \, | \,
                 r \ (r\in \R) \,  | \,
                 \op(\mathsf{tr}) \, | \, 
                 \texttt{let} \, x = m \, \texttt{in} \, n \, | \, 
                 * \, |\, (\mathsf{tr},\mathsf{tr}) \, | \, \fst(\mathsf{tr}) \, | \, \snd(\mathsf{tr})
\]

Abadi and Plotkin also defined a sublanguage of trace terms called values.

\begin{align*}
  \mathsf{v} &:= x \, | \, r \ (r\in \R)  \, | * \, | \, (v,v) \\
  \mathsf{v}_\mathsf{bool} &:= \true \, |\, \false
\end{align*}

The operational semantics of a program then consists of two mutually inductively defined reductions:
symbolic evaluation and ordinary evaluation -- the former yields a trace term and the latter yields 
a value.  Then the main idea is that to evaluate a term, when you hit a reverse differential, 
$v.\rd(x.f)(a)$, you evaluate $f$ symbolically, just enough to remove control structures and derivatives 
giving a trace term.  And then this trace term is differentiated symbolically, yielding a trace 
term, and the evaluation continues.  

Note that defining symbolic reverse differentiation does not require any evaluation functions.  However, 
we do at this point require, as \cite{journal:Abadi-Plotkin} did, that for each function symbol 
$\op \in \Sigma(T,U)$ there is an associated a function symbol $\op_R\in \Sigma(T\x U,T)$.  The idea is that $\op_R$ 
is the reverse derivative of $\op$.  We will write $v.\op_R(a)$ as notation for $\op_R(a,v)$.  Then 
define symbolic reverse differentiation $v.\Rd(x.f)(a)$ by induction over trace terms $f$ and where $v$ and $a$ are values:
\begin{align*}
  w.\Rd(x.y)(a) &= \begin{cases} w & x=y \\ 0 & x\ne y \end{cases}\\
  w.\Rd(x.r)(a) &= 0 \qquad r \in \R \\
  w.\Rd(x.m+n)(a) &= w.\Rd(x.m)(a) + w.\Rd(x.n)(a) \\
  w.\Rd(x.\op(m))(a) &= \texttt{let}\, x=a,t=w.\op_R(m) \, \texttt{in}\, t.\rd(m)(a) \qquad t \text{ fresh}\\
  w.\Rd(x.\texttt{let}\, y=d\, \texttt{in}\, e)(a)
                 &= \texttt{let}\, x=a ,y=d\, \texttt{in}\, w.\rd(x.e)(a) \\
                &+ (\texttt{let}\, t = w.\rd(y.e)(y) \,\texttt{in}\, t.\rd(x.d)(a)) \qquad t \text{ fresh}\\
  w.\Rd(*)     &= 0\\
  w.\Rd(x.(u,v))(a) &= \texttt{let}\, (y,z) = w \, \texttt{in}\, y.\Rd(x.u)(a) + z.\Rd(x.v)(a)\\
  w.\Rd(x.\fst(m))(a) &= \texttt{let}\, x=a,t=m \, \texttt{in} \, (w.0).\Rd(x.m)(a) \qquad t \text{ fresh}\\
  w.\Rd(x.\snd(m))(a) &= \texttt{let}\, x=a,t=m\, \texttt{in} \, (0,w).\Rd(x.m)(a) \qquad t \text{ fresh}
\end{align*}
The let term is also the chain rule but for differentiating with 
respect to the two variable function $\Gamma,x,y \proves n$, so that we get the usual rule
$\partial n/\partial t = \partial n/\partial x\cdot \partial x/\partial t + \partial n/\partial y\cdot \partial y/\partial t$
appropriately reversed.

Also, for the projection rule, one might have expected just $(w,0).\Rd(m)(a)$.  Under  interpretation 
we certainly get a term of the form $R[a\pi_0] = (\rs{a} \x \iota_0)R[a]$.  However by {\bf [RD.8]},
$(\rs{a} \x 1)R[a] = R[a]$. We will see below that if our evaluation satisfies a certain 
property, then the simpler translation is warranted.

Then as long as our interpretation always sends $\op_R$ to the reverse derivative of $\op$, then 
symbolic and formal reverse differentiation agree under interpretation.

\begin{proposition}[Symbolic differentiation correctness]\label{proposition:symb-correct}
  Suppose $\X$ is a basic reverse differential join restriction category, and suppose that we have a fixed interpretation 
  of $\sdpl$ into $\X$ for which $\den{\op_R} = R[\den{\op}]$ then for all values $a,v$ and for all traced terms 
  $m$ 
  \[
    \den{v.\rd(x.m)(a)} = \den{v.\Rd(x.m)(a)}
  \]
\end{proposition}

We have an analogous proposition for the interpretation of all $\sdpl^+$.  

\begin{proposition}[Symbolic differentiation correctness extended]\label{proposition:symb-correct-extended}
  Suppose $\X$ is a basic reverse differential join restriction category, and that we have a fixed interpretation 
  of $\sdpl^+$ into $\X$ for which $\den{op_R} = R[\den{\op}]$, then for all values $a,v$ and for all traced terms 
  $m$:
  \[
   \den{v.\rd(x.m)(a)}_\phi = \den{v.\Rd(x.m)(a)}_\phi
  \]
\end{proposition}

We then define the operational semantics of $\sdpl$ exactly as done by Abadi and Plotkin \cite{journal:Abadi-Plotkin}: an \textbf{operational
structure} is given by $(\ev,\bev,{\underline{~~}}_R)$ where 
\[
  \ev_{T,U} : \Sigma(T,U) \x \mathsf{v}_T \to \mathsf{v}_U  \qquad 
  \bev_T : \mathsf{Pred}(T) \x \mathsf{v}_T \to \mathsf{val}_\mathsf{bool}
\]
are partial functions.  We denote closed value terms $\mathsf{v}$ that have type $Y$ as
 $\mathsf{v}_Y$ and 
the set of closed $\mathsf{v}_\mathsf{bool}$ as $\mathsf{val}_\mathsf{bool}$ (these sets are precisely those that require formation 
in an empty context $\proves m:A$ and $\proves b$).  Further 
\[
  \underline{~~}_R : \Sigma(T,U) \to \Sigma(T\x U,T) \qquad \op \mapsto \op_R
\]
With these three pieces one may define ordinary reduction $\Rightarrow$ from terms to values and symbolic reduction $\rightsquigarrow$
from terms to trace terms by induction; see \cite{journal:Abadi-Plotkin} for details.  For $\sdpl$ these reduction relations 
are formulated with respect to a value environment: this is a mapping of variable names to 
closed value terms.  For $\sdpl^+$ we also require a function environment: this is a mapping $\varphi$ of 
function names to closures.  A closure is a tuple $(\varphi,f,x,m)$ where $m$ has at most the free 
ordinary variable $x$, and additionally, all the free function variables in $m$ except $f$ are in the domain 
of $\varphi$.  The idea is that closures are created when evaluating $\texttt{letrec} \, f(x) := m \, \texttt{in}\, n$; 
if our current function environment is $\varphi$ we extend it with $(\varphi,f,x,m)$ and continue 
evaluating $n$ -- this way if $n$ calls $f$ then the definition of $f$ can be looked up in the 
function environment, and any symbol that the body of $f$ requires to operate will be there too.

\section{Denotational Semantics}\label{sec:den-sem}
An interpretation structure $(A \in \X_0,(1\to^{a_r}A)_{r\in \R},\den{\blank},\den{\blank}_T,\den{\blank}_F)$
is a \textbf{differentially denotational interpretation structure} when 
\begin{enumerate}
  \item For all closed value terms of $\mathsf{v}$ we have that $1 \to^{\den{\mathsf{v}}} \den{A}$ is a total point of $\den{A}$;
  \item For all $\op \in \Sigma$ we have $R[\den{\op}] = \den{\op_R}$;
  \item For all closed value terms $v \in \mathsf{v}_A$ we have 
        \[
          \begin{tikzcd}
            1 \ar[r,"\den{v}"] \ar[dr,"\den{\ev(\op,v)}"'] & \den{A}\ar[d,"\den{\op}"]\\
            & \den{B}
          \end{tikzcd}
          \qquad 
          \begin{tikzcd}
            1 \ar[dr,"\den{\bev(\mathsf{pred},v)}_H"'] \ar[r,"\den{v}"] & \den{A} \ar[d,"\den{op}_H"]\\
            & 1
          \end{tikzcd}
        \]
        where $H$ is either $T$ or $F$.  In particular, both sides may be undefined, but they must be undefined simultaneously.
\end{enumerate}

The idea behind showing that a denotational semantics captures a language's operational semantics is that if $m \Rightarrow v$ then $\den{m} = \den{v}$.  However, the operational semantics for $\sdpl$ and $\sdpl^+$
is defined with respect to value and function environments, and we have two operational relations.  
Interpreting a term $m$ with free variables $x_1,\ldots,x_n$ in a value environment 
$\{x_i := v_i\}_{1 \leq i \leq }$ is straightforward: since each $v_i$ is a closed term, first interpret 
$m$ as above: $\den{\Gamma} \to^{\den{m}} \den{B}$, and then precompose with the point of $\den{\Gamma}$ 
given by $1 \to^{\<\den{v_i}\>_{i\leq n}} \den{\Gamma}$.  Next we need the following lemma:

\begin{lemma}
The interpretation of terms of $\sdpl^+$ extends to allow the construction of an element of $\den{\Phi}$ for 
each function environment $\varphi$ whose domain is $\Phi$.
\end{lemma}

Note that for any trace term $c$ it always fully evaluates.  It requires no function context because 
it has no function symbols, and we have that for any value environment $\rho$, $c \Rightarrow v$ for 
some closed value term $v$.  

The goal is then to prove the following theorem by mutual induction: for any term $m$, 
any value environment $\rho$, and function environment $\varphi$, we have that if 
$m \rightsquigarrow c$ then 
\[
  \den{m} = \den{c} = \den{v}.  
\]

\section{Potential operational improvements}

In this section we describe additional properties our categorical semantics has that may lead to a more refined 
operational semantics.

The compatibility between differentiation and restriction: \textbf{[RD.8,9]} state essentially that the 
definedness of the reverse derivative of a term is completely determined by the term itself.  This is 
relevant to a more efficient semantics: the operational semantics used here has the property that when 
taking the reverse derivative over looping or recursive constructs, we first build a trace term, 
which turns out to be a (long) series of \texttt{let} expressions that describe the evolution of the 
state of the computation.  We then symbolically differentiate these \texttt{let} expressions which 
always results in the creation of a sum of two expressions for each such \text{let} expression -- and the 
number of \text{let} expressions created by recursion or looping is the number of times that the 
function recursed or the number of times the loop ran.  Thus we quickly get wide trees of sums 
of symbolic terms that need to be evaluated.  However, at each step of this process, one of these 
terms is of the form $v.\rd(x.m)(a)$ where $x$ does not occur freely in $m$, and hence can 
be proven to always evaluate to $0$ if it evaluates to anything.  Our semantics has the following property 

\begin{lemma}
  For any term $m$ in which $x$ does not occur 
  \[
    \den{v.\rd(x.m)(a)} = \rs{\<1,\<\den{a},\den{v}\>\>\den{m}}0
  \]
\end{lemma}

And moreover, the let expressions that get zeroed out have all their subterms occuring in the term 
that does not get zeroed out.  We then have the following lemma
 \begin{lemma}
  If we added the rule 
  \[
    x \not \in \fv(e) \Rightarrow w.\Rd(x.\texttt{let}\, y=d\, \texttt{in}\, e)(a) 
    := \texttt{let}\,x=a,y=d,t=w.\rd(y.e)(y) \, \texttt{in}\, t.\rd(x.d)(a)
  \]
  Then Propositions \ref{proposition:symb-correct} and \ref{proposition:symb-correct-extended} would 
  still hold.
 \end{lemma}

This gives an operational semantics where differentiating over looping constructs does not 
 have a branching blowup, and hence experiences an exponential speedup.

 Reverse differential restriction categories, as we have seen earlier,
 allow forming a forward derivative from the reverse derivative.  They also allow 
 forming a reverse derivative from that forward derivative.  In a reverse differential restriction 
 category,  {\bf [RD.6]} is equivalent to the requirement that the process of going from a 
 reverse derivative to a forward derivative and then back to a reverse derivative gives 
 exactly the starting reverse derivative.

 \begin{lemma}
  For any map from $A \x B \to^{f} C$ define a map $A \x C\to ^{f^{\dagger[A]} := (\iota_0 \x 1)R[f]\pi_1} B$.
  We always get a forward derivative as $D[f] := R[f]^{\dagger[A]}$.  Then {\bf [RD.6]} is equivalent to 
  requiring that $D[f]^{\dagger[A]} = R[f]$.
 \end{lemma}

 This kind of coherence for defining forward derivatives from their reverse could be useful in 
 using the forward derivative and then converting back via daggering the result.  

 \begin{lemma}
  For any operational symbol $\op$ if the evaluation function used by the operational semantics satisfies 
  \[
    \mathsf{eval}(\snd(\op_{RRR},(((a,0),0),(0,b)))) = \mathsf{eval}(\op_R,(a,b))
  \]
  then this can be modelled in any reverse differential restriction category.  Moreover, for every term $m$ 
  we have 
  \[
   \den{\rd(x.\rd(y.\rd(z.m)(a).y)(b).x)(c).w} = \rs{\den{b}}\rs{\den{c}}\den{\rd(z.m)(a).w}
  \]
 \end{lemma}
 Crucially for the above, we require \textbf{[RD.6]}.  

 An aspect of forward differentiation that is modelled in our semantics is that differentiating a differential with respect to its ``direction'' 
 is just substitution.  That is $\fd(x.  \fd(y.m)(a).x  )(b).v = \fd(y.m)(a).v$ is modelled.  This uses \textbf{[RD.6]}.  More generally, we can 
 modify the type system slightly to keep track of the arguments that a term is differentiated to by introducing another context, which we call 
 a linearity context.  Then the typing judgment for the reverse differential term would have two forms:
 \[
   \infer{\Gamma,a:A|\Delta,v:B \proves \rd{x.m}(a).v : A \qquad a,v \text{ fresh}}{\Gamma,x:A|\Delta \proves m:B}
   \qquad 
   \infer{\Gamma,a:A|\Delta,v:B \proves \rd{x.m}(a).v : A \qquad a,v \text{ fresh}}{\Gamma|\Delta,x:A \proves m:B}
 \]
 And if we are forward differentiating with respect to a variable from the linearity context: i.e., if 
 we form the forward derivative of a term with respect to a variable from the linearity context; i.e., 
 if $v$ was in the linearity context of a term $m$ and we form $\fd(v.m)(a).w)$, then operational reduction
 \[
     \fd(v.m)(a).w \rightsquigarrow \texttt{let}\, v = w \, \texttt{in}\, m 
 \]
 is modelled in our semantics.  This means that we can completely avoid doing differentiation in 
 some cases, at the cost of having to carry around more type information.  There is a similar version 
 of this rule for reverse derivatives and it has to do with ``colet'' expressions.  In $\sdpl$ 
 we can use the reverse derivative to create a term that substitutes linearly into the 
 output variable of a term.  We could use these ``colet'' expressions 
 and allow for speedups of reverse derivatives as well.  It might also be interesting 
 to characterize these constructions in their own right.  This approach also allows us to  force
  \textbf{[RD.6]} into the operational semantics. 
 
 The axiom \textbf{[RD.7]}, dealing with the symmetry of mixed partial derivatives, 
 may also have a role to play in simplifying the operational semantics.  
 Some machine-learning algorithms use the Hessian of the error function to optimize backpropagation itself, 
 allowing for both more efficient and effective training (for one example, see 
 \cite{proceedings:neural_net_hessian,thesis:visual-dp-thesis-of-the-year-2020}).  
These second derivatives are expected to satisfy a 
 higher dimensional analog of the chain rule.  In fact one might expect in general higher analogs of the 
 chain rule to hold, which are sometimes called the Faa di Bruno formulae for higher chain rules on 
 terms of the form $\partial^n (fg)$.  These expected formulae will all hold in our semantics due to a
 result that shows \textbf{[CD.6,7]} are equivalent to having all the Faa di Bruno formulae \cite{journal:Cockett-Seely:Faa}.  These 
 higher chain rule expansions can be used to determine a slightly different operational semantics for 
 $\rd(x.m)(a).v$ expressions, where the chain rule is maximally expanded first, and linearity reductions occur, 
 and then symbolic differentiation is used.  While it is unclear if this is more efficient, it would make things simpler, as it would
 guarantee that the operational semantics captured the higher chain rule formulae without having to 
 make a requirement of the evaluation function on $\op_{RRR}$.

\bibliography{references}

\end{document}